\documentclass{article}
\usepackage{amsmath,amssymb,latexsym,hyperref,url,amsthm,enumerate}
\usepackage{amsfonts}
\usepackage{authblk}
\usepackage{tikz,pgfplots}
\pgfplotsset{compat=1.10}

\newtheorem{cor}{Corollary}
\newtheorem{definition}{Definition}
\newtheorem{lemma}{Lemma}
\newtheorem{prop}{Proposition}
\newtheorem{theorem}{Theorem}
\newtheorem{example}{Example}
\newtheorem{conjecture}{Conjecture}

\newcommand{\Av}{\operatorname{Av}}

\makeatletter
\def\ps@pprintTitle{%
 \let\@oddhead\@empty
 \let\@evenhead\@empty
 \def\@oddfoot{}%
 \let\@evenfoot\@oddfoot}
 \def\keywords{\xdef\@thefnmark{}\@footnotetext}

\title{Two-stack-sorting with pop stacks}

\author[1]{Lara Pudwell}
\author[2]{Rebecca Smith}

\affil[1]{Department of Mathematics and Statistics, Valparaiso University, Valparaiso, IN, USA}
\affil[2]{Department of Mathematics, SUNY Brockport, Brockport, NY, USA}

\begin{document}

\maketitle

\begin{abstract}
We consider the set of permutations that are sorted after two passes through a pop stack.  We characterize these permutations in terms of forbidden patterns (classical and barred) and enumerate them according to the ascent statistic.  Then we show these permutations to be in bijection with a special family of polyominoes.  As a consequence, the permutations sortable by this machine are shown to have the same enumeration as three classical permutation classes.
\end{abstract}






\section{Introduction}\label{Introduction}

In this paper we study the permutations that are sortable by two passes through a pop stack.  We first introduce necessary definitions and notation and give a survey of related results.  In Section~\ref{S:patterns} we characterize the two-pop-stack sortable permutations, and in Section~\ref{S:enum} we enumerate the permutations according to the number of ascents.  The enumeration shows that the number of such permutations follows a linear recurrence with constant coefficients, so we give a second enumeration argument that reflects this recursive structure.  In Section~\ref{S:poly} we show these permutations to be in bijection with a special family of polyominoes.

We also note that pop stacks can be used to model genome rearrangements, as the most common rearrangement on genomes is reversal.  Rather than the traditional greedy model which reverses only one decreasing sequence at a time, a pop stack reverses all maximal decreasing subsequences of the permutation at each stage.  Our particular algorithm for networking pop stacks makes this greedier algorithm apply at each pass.  While certainly not an optimal algorithm, the structure involved with pop stacks makes it easier to handle multiple reversals at once.  

\subsection{Permutations}

Let $\mathcal{S}_n$ be the set of permutations of $[n]=\{1,2,\dots,n\}$.  Given $\pi \in \mathcal{S}_n$ and $\rho \in \mathcal{S}_k$ we say that $\pi$ \emph{contains} $\rho$ as a pattern if there exist $1 \leq i_1 < \cdots < i_k \leq n$ such that $\pi_{i_a} < \pi_{i_b}$ if and only if $\rho_a < \rho_b$; in this case we say that $\pi_{i_1}\pi_{i_2}\cdots\pi_{i_k}$ is \emph{order-isomorphic} to $\rho$.  Otherwise, $\pi$ \emph{avoids} $\rho$.  Alternatively, let the \emph{reduction} of the word $w$, denoted $\operatorname{red}(w)$, be the word formed by replacing the $i$th smallest letter of $w$ with $i$.  Then $\pi$ contains $\rho$ if there is a subsequence of $\pi$ whose reduction is $\rho$.

\begin{example}  The permutation $\pi=35841726$ contains the permutation $\rho =3241$ since the reduction of the subsequence $5472$ is $\operatorname{red}(5472)=3241$.
\end{example}
Our results also require a second kind of permutation pattern.  

\begin{definition}  A \emph{barred pattern} is a permutation $\rho \in \mathcal{S}_k$ where any entry may have a bar over it.   A permutation $\pi$ is said to \emph{contain} $\rho$ if $\pi$ contains a permutation pattern made up of the non-barred entries of $\rho$ that does not extend to a permutation pattern including all entries of $\rho$.
\end{definition}
As such, permutation $\pi$ \emph{avoids} a barred pattern $\rho$ if each copy of the pattern consisting of non-barred entries of $\rho$ in $\pi$ extends to a copy of the permutation made up of all entries of $\rho$.

\begin{example}  The permutation $\pi=35841726$ avoids the permutation $\rho =3\overline{5}241$ since the only occurrence of $3241$ (realized by $5472$)  is part of an occurrence of the pattern $35241$ (realized by $58472$).
\end{example}

Given two permutations $\alpha \in \mathcal{S}_j$ and $\beta \in \mathcal{S}_{\ell}$, the \emph{direct sum}, denoted $\alpha \oplus \beta$ is the permutation resulting from the concatenation of $\alpha$ with 
$\beta$ where all digits of $\beta$ are incremented by $j$.  For example, $321 \oplus 1 \oplus 21 \oplus 321 = 321465987$.

An \emph{ascent} of permutation $\pi$ is an index $i$ where $\pi_i < \pi_{i+1}$, while a \emph{descent} is an index $i$ where $\pi_i > \pi_{i+1}$.  We denote the number of ascents of $\pi$ by $\mathrm{asc}(\pi)$ and the number of descents by $\mathrm{des}(\pi)$.

\begin{definition}~\label{graph_def}  The \emph{graph} of a permutation $\pi \in \mathcal{S}_n$ is the set of points $$\left\{(i,\pi_i) \middle| 1 \leq i \leq n\right\}.$$
\end{definition}
For example, the graph of 321465987 is given in Figure~\ref{F:graph}.

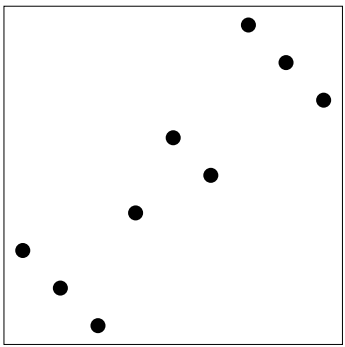
\begin{figure}
\begin{center}
\scalebox{0.5}{\begin{tikzpicture}
					\draw (0,0) rectangle (9,9);		
					\fill[black] (.5,2.5) circle (.2cm);
					\fill[black] (1.5,1.5) circle (.2cm);
					\fill[black] (2.5,0.5) circle (.2cm);
					\fill[black] (3.5,3.5) circle (.2cm);
					\fill[black] (4.5,5.5) circle (.2cm);
					\fill[black] (5.5,4.5) circle (.2cm);
					\fill[black] (6.5,8.5) circle (.2cm);
					\fill[black] (7.5,7.5) circle (.2cm);
					\fill[black] (8.5,6.5) circle (.2cm);
\end{tikzpicture}}
\end{center}
\caption{The graph of $\pi=321465987$}
\label{F:graph}
\end{figure}

\subsection{Sorting Networks}

A stack is a last-in first-out data structure with push and pop operations.  Knuth \cite{K73} studied permutations that are sortable by one pass through a stack; in other words, there is a sequence of push and pop operations to transform the permutation $\pi \in \mathcal{S}_n$ into the increasing permutation $1\cdots n$ as output.  Knuth showed that a permutation is sortable by one pass through a stack if and only if $\pi$ avoids the pattern 231.  There are $C_n$ such permutations of length $n$, where $C_n=\dfrac{\binom{2n}{n}}{n+1}$ is the $n$th Catalan number.  Other researchers have studied networks with multiple stacks in series or in parallel, including \cite{EI71, P73, T72}.

Let $S(\pi)$ be the output from passing $\pi$ through a single stack.  Knuth's result shows that $S(\pi)=12\cdots n$ if and only if $\pi$ avoids 231.  If we keep the convention that the stack must be increasing from top to bottom, then $S(\pi)$ is well-defined.  We push a new element onto the stack when the stack is empty or when the next available input is smaller than the top element of the stack.  We pop an element to output when the top element of the stack is smaller than the next available input or when the input is empty.  With this convention, $S(1)=1$ and for $n > 1$, $S(\pi_1\cdots \pi_{i-1} n \pi_{i+1}\cdots \pi_n) = S(\pi_1\cdots \pi_{i-1})S(\pi_{i+1}\cdots \pi_n)n$.  West \cite{W90} defined two-stack-sortable permutations as those for which $S(S(\pi))=12\cdots n$.  He showed that a permutation is sortable by two passes through a stack if and only if $\pi$ avoids 2341 and $3\overline{5}241$, and  Zeilberger \cite{Z92} showed that there are $\dfrac{2(3n)!}{(n + 1)!(2n + 1)!}$ such permutations of length $n$.  

Notice that West's definition is not the most efficient sorting algorithm since it does not look ahead to use the second pass through the stack strategically.  However, in addition to requiring only linear time to implement, this approach also never creates new \emph{inversions} along the way.  That is, if entries $\pi_i$ and $\pi_j$ are in the correct (i.e., increasing) relative order at some stage in the stack sorting process, they will remain that way in all future iterations.  We also note that this sorting algorithm is distinct from sorting with stacks in parallel or in series.  

In this paper, we consider the analogous characterization and enumeration results for pop stacks.  A pop stack is a stack where the only way to move an element from the stack to the output is to pop everything in the stack (in last-in first-out order).  That is, a pop stack can reverse the maximal contiguous decreasing subsequences of a permutation.  Both Avis and Newborn \cite{AN81} and Atkinson and Stitt \cite{AS02} studied pop stacks in series.  Atkinson and Sack \cite{AS99} and Smith and Vatter \cite{SV09} also considered pop stacks in parallel.  It follows from the work of Avis and Newborn that a permutation $\pi$ is sortable by one pass through a pop stack if and only if $\pi$ avoids 231 and 312.  There are $2^{n-1}$ such permutations. Permutations that avoid 231 and 312 are known as \emph{layered} permutations since they are the direct sum of decreasing permutations.  Further, layered permutations of length $n$ are in bijection with compositions (that is, ordered integer partitions) of $n$ since these permutations are uniquely determined by the lengths of the layers.  

\begin{example}  The permutation $321465987$, whose graph is shown in Figure~\ref{F:graph}, is a layered permutation with layers of size $3, 1, 2,$ and $3$, so it corresponds to the composition $3+1+2+3$.
\end{example}

\section{Two Pop Stacks}\label{S:2PS}

Our main concern is permutations which are sortable by two passes through a pop stack.  Let $P(\pi)$ be the output from running $\pi$ through a single pop stack. Keeping the convention of West, if the stack is increasing from top to bottom, then $P(\pi)$ is well-defined.  Let $\pi_1\cdots \pi_i$ be the longest decreasing prefix of $\pi \in \mathcal{S}_n$.  Then $P(1)=1$ and for $n>1$, $P(\pi) = \pi_i \cdots \pi_1 P(\pi_{i+1}\cdots \pi_n)$.  If $P(P(\pi))=12\cdots n$, we say that $\pi$ is two-pop-stack sortable and write $\pi \in \mathcal{P}_{2,n}$.  Further, we let $\mathcal{P}_2 = \bigcup_{n \geq 0} \mathcal{P}_{2,n}$.  We characterize and enumerate the permutations in $\mathcal{P}_{2,n}$ below.  Both results rely on the following definition and lemma.

A \emph{block} of a permutation is a maximal contiguous decreasing subsequence.  For example if $\pi=21534$, there are three blocks: $B_1=21$, $B_2=53$, and $B_3=4$.  Conceptually, a block is a set of letters that get output at the same time when we run $\pi$ through a pop stack.  Blocks characterize $\mathcal{P}_{2}$ in the following way:

\begin{lemma}\label{L:blocks}
Let $\pi$ be a permutation with blocks $B_1, \dots, B_\ell$.  Then, $\pi$ is two-pop-stack sortable if and only if for $1 \leq i \leq \ell-1$, $\max(B_i) \leq \min(B_{i+1})+1$.
\end{lemma}

\begin{proof}
Suppose $\pi$ has blocks $B_1, \dots, B_\ell$.  By definition, each block consists of a decreasing sequence of elements and so $\max(B_i)$ is the first element in block $i$ while $\min(B_{i+1})$ is the last element in block $i+1$.  By definition of $P(\pi)$, $\max(B_i)$ and $\min(B_{i+1})$ are adjacent letters in $P(\pi)$.

If $\pi \in \mathcal{P}_2$, then $P(\pi)$ is layered.  This means that adjacent elements in $P(\pi)$ have one of two relationships.  Either they form an ascent (in which case $\max(B_i)<\min(B_{i+1})$) or they form a descent.  If two letters form a descent in a layered permutation, they must have consecutive values, that is $\max(B_i) = \min(B_{i+1})+1$.
\end{proof}

Figure~\ref{F:onesort} gives an illustration of Lemma~\ref{L:blocks} using the graph (as defined in Definition~\ref{graph_def}) of the permutation $\pi=215364$ before and after one run through a pop stack.  Blocks 1 and 2 show the behavior where $\max(B_1)<\min(B_{2})$ while blocks 2 and 3 have $\max(B_2) = \min(B_{3})+1$. In either event, applying the pop stack algorithm causes the consecutive blocks to form layered subpermutations.  Hence $P(\pi)$ is a layered permutation. 

\begin{figure}
\begin{center}
\begin{tabular}{cc}
\scalebox{0.5}{\begin{tikzpicture}
\draw (0,0)--(6,0)--(6,6)--(0,6)--(0,0);
\draw (2,0)--(2,6);
\draw (4,0)--(4,6);
\fill[black] (.5,1.5) circle (.2cm);
\fill[black] (1.5,0.5) circle (.2cm);
\fill[black] (2.5,4.5) circle (.2cm);
\fill[black] (3.5,2.5) circle (.2cm);
\fill[black] (4.5,5.5) circle (.2cm);
\fill[black] (5.5,3.5) circle (.2cm);
\node at (1,6.6) {\Huge$B_1$};
\node at (3,6.6) {\Huge$B_2$};
\node at (5,6.6) {\Huge$B_3$};
\end{tikzpicture}}
&
\scalebox{0.5}{\begin{tikzpicture}
\draw (0,0)--(6,0)--(6,6)--(0,6)--(0,0);
\draw (2,0)--(2,6);
\draw (4,0)--(4,6);
\fill[black] (.5,0.5) circle (.2cm);
\fill[black] (1.5,1.5) circle (.2cm);
\fill[black] (2.5,2.5) circle (.2cm);
\fill[black] (3.5,4.5) circle (.2cm);
\fill[black] (4.5,3.5) circle (.2cm);
\fill[black] (5.5,5.5) circle (.2cm);
\node at (1,6.6) {\Huge $B^{\text{rev}}_1$};
\node at (3,6.6) {\Huge $B^{\text{rev}}_2$};
\node at (5,6.6) {\Huge $B^{\text{rev}}_3$};
\end{tikzpicture}}\\
$\pi$&$P(\pi)$\\
\end{tabular}
\end{center}

\caption{The permutation 215364 before and after one pass through a pop stack}
\label{F:onesort}
\end{figure}
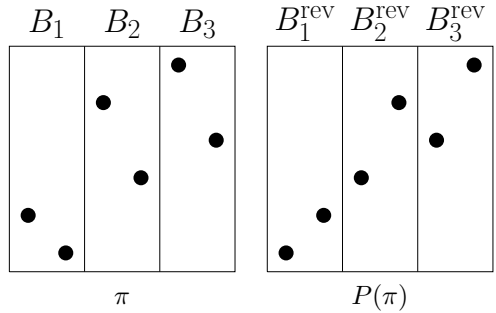

We also introduce divided notation for permutations.  This was used specifically in reference to a pop stack by the second author and Vatter~\cite{SV14} where divisions represented by vertical lines indicate when entries are to be popped at intermediate stages while attempting to sort the permutation $\pi$ by a pop stack.  The algorithm used here forces a division between $\pi_i$ and $\pi_{i+1}$ exactly when $\pi_i$ and $\pi_{i+1}$ form an ascent, that is, at the end of each block.  

At each division, the entries are reversed by the first run through the pop stack.  Thus a two-pop-stack sortable permutation $\pi$ must avoid the divided permutations $2 | 3 | 1, 2 | 13, 32 | 1,132$ as these are precisely the permutations that cause $P(\pi)$ to contain a $231$ pattern.  Similarly, $\pi$ must avoid  $3 | 1 | 2, 3 | 21,  13 | 2, 213$ as these are exactly the permutations that cause $P(\pi)$ to contain a $312$ pattern.   As there must be a division between any two entries forming an ascent, we obtain the following lemma. 

\begin{lemma}~\label{L:divide}
A permutation $\pi$ is two-pop-stack sortable if and only if  $\pi$ avoids
$2 | 3 | 1,  32 | 1,  3 | 1 | 2, 3 | 21$ when $\pi$ is written with divisions between all ascents.
\end{lemma}

\subsection{Characterization}\label{S:patterns}

In Lemma~\ref{L:blocks} and Lemma~\ref{L:divide}, we characterized two-pop-stack sortable permutations in terms of blocks and divided permutations.  Here, we characterize these permutations in the more conventional language of pattern avoidance.  Lemma~\ref{L:blocks} and Lemma~\ref{L:divide} are equivalent to the following theorem.

\begin{theorem}\label{T:patterns}
Permutation $\pi$ is two-pop-stack sortable if and only if $\pi$ avoids the patterns $2341, 3412, 3421,
4123, 4231, 4312, 4\overline{1}352$, and $413\overline{5}2$.
\end{theorem}

\begin{proof}
Suppose that $\pi$ is not two-pop-stack sortable.  By Lemma~\ref{L:blocks}, there exist two adjacent blocks of $\pi$, $B_i$ and $B_{i+1}$, such that $\max(B_i) \geq \min(B_{i+1})+2$.  Let $a=\max(B_i)$ and $b=\min(B_{i+1})$.  Clearly $a>b$.  Further since $a$ and $b$ are in different blocks, there must be one ascent between them; that is, $a$ and $b$ are the first and last letters of either a 231 pattern, a 312 pattern, or a 4231 pattern.

If $a$ and $b$ are the first and last letters in a 231 pattern and $a \geq b+2$, there must be another digit $c$ such that $a>c>b$.  If $c$ appears before $a$, then $c$ together with the 231 pattern forms a 2341 pattern.  If $c$ appears in block $B_i$, then $c$ together with the 231 pattern forms a 3241 pattern.  If $c$ appears in block $B_{i+1}$, then $c$ together with the 231 pattern forms a 3421 pattern.  If $c$ appears after $b$, then $c$ together with the 231 pattern forms a 3412 pattern.

If $a$ and $b$ are the first and last letters in a 312 pattern and $a \geq b+2$, there must be another digit $c$ such that $a>c>b$.  If $c$ appears before $a$, then $c$ together with the 312 pattern forms a 3412 pattern.  If $c$ appears in block $B_i$, then $c$ together with the 312 pattern forms a 4312 pattern.  If $c$ appears in block $B_{i+1}$, then $c$ together with the 312 pattern forms a 4132 pattern.  If $c$ appears after $b$, then $c$ together with the 312 pattern forms a 4123 pattern.

If $a$ and $b$ are the first and last digits in a 4231 pattern, then $a \geq b+2$ already.

Therefore, if there exist two adjacent blocks of $\pi$, $B_i$ and $B_{i+1}$, such that $\max(B_i) \geq \min(B_{i+1})+2$, then $\pi$ contains at least one of 2341, 3241, 3412, 3421, 4123, 4132, 4231, or 4312 as a pattern.  However, the digits serving as $a$ and $b$ in are in adjacent blocks. If we have a 3241 pattern where $a$ plays the role of `3' and $b$ plays the role of `1', then there can be no letter less than $b$ that appears between `3' and `2' as there is only one ascent, namely from $B_i$ to $B_{i+1}$.  In other words, $\pi$ contains a copy of 4352 (i.e. 3241) that does not extend to a 41352 pattern; that is $\pi$ contains $4\overline{1}352$.  Similarly, if we have a 4132 pattern where $a$ plays the role of `4' and $b$ plays the role of `2', then there can be no letter greater than $a$ that appears between `3' and `2'.  In other words, $\pi$ contains a copy of 4132 that does not extend to a 41352 pattern; that is, $\pi$ contains $413\overline{5}2$.

Therefore, if $\pi$ is not two-pop-stack-sortable, $\pi$ contains at least one of the patterns 2341, 3412, 3421, 4123, 4231, 4312, $4\overline{1}352$, and $413\overline{5}2$.

To show the converse, recall Lemma~\ref{L:divide} characterizes two-pop-stack sortable permutations by the avoidance of
$2 | 3 | 1,  32 | 1,  3 | 1 | 2, 3 | 21$ when permutations are written with divisions between each descent.

The patterns 2341, 3412, 3421, 4123, 4231, 4312, $4\overline{1}352$, and $413\overline{5}2$, must have at least the following divisions:
\[
2 | 3 | 41, \; 3 | 41 | 2, \;3 | 421, \;41 | 2 | 3, \;42 | 31, \;431 | 2, \;4\overline{1}3 | 52, \; \text{and} \;41 | 3\overline{5}2\; \text{respectively}.
\]
Notice that no additional divisions can prevent $2 | 3 | 41$ from containing $2 | 3 | 1$ nor can they prevent $3 | 41 | 2,  41 | 2 | 3 $ from containing $3 | 1 | 2$.  The division needed to prevent $3 | 421$ or $42 | 31$ from containing a $3 | 21$ pattern forces the new containment of a $2 | 3 | 1$ pattern.  And the division needed to prevent $431 | 2$ from containing a $32 | 1$ forces the new containment of a $ 3 | 1 | 2$ pattern.

Finally consider the barred patterns.  Notice $4\overline{1}3 | 52$ with no additional divisions contains a $32 | 1$ pattern.  An additional division between the `4' and `3'  indicates an ascent occurred between these entries (possibly involving one of them).  An entry larger than the `4' yields a $3421$ subpattern and an entry less than the `3', but greater than the `2' yields a $4231$ subpattern.  Further, an entry with value between `3' and `4' does not cause an ascent to occur and creates an additional need for a division for the same reason as above.  Only an entry smaller than the `2' will allow for the divided permutation $41 | 3 | 52$ which is two-pop-stack sortable.  A similar argument shows permutations containing $41 | 3\overline{5}2$ are not two-pop-stack sortable.
\end{proof}

\subsection{Enumeration}\label{S:enum}

Next, we determine $\left|\mathcal{P}_{2,n}\right|$.   By definition, a permutation has an ascent at position $i$ exactly when $\pi_i$ and $\pi_{i+1}$ are in different blocks.  Therefore, the number of blocks of $\pi$ is one more than the number of ascents of $\pi$.  In light of Lemma~\ref{L:blocks} it is natural to consider two-pop-stack sortable permutations with a fixed number of ascents.

\begin{prop} Let $a(n,k) = \left|\left\{\pi \in \mathcal{P}_{2,n} \middle| \mathrm{asc}(\pi)=k \right\}\right|$ and let \\$b(n,k) = \left|\left\{\pi \in \mathcal{P}_{2,n} \middle| \mathrm{asc}(\pi)=k \text{ and the last block of } \pi \text{ has size 1}\right\}\right|$. \\For $n \geq 0$:
\begin{align*}
a(n,k)&=\begin{cases}
1& k=0 \text{ or } k=n-1 \\
0& k<0 \text{ or } k \geq n \geq 1\\
2 \displaystyle{\sum_{i=1}^{n-1} a(i,k-1)} - b(n-1,k-1)&\text{ otherwise}
\end{cases}\\
\text{and}\\
b(n,k)&=\begin{cases}
1& k=n-1\\
0& k<1 \text{ or } k \geq n\\
2 a(n-1,k-1) - b(n-1,k-1)&\text{ otherwise.}
\end{cases}
\end{align*}
\label{P:recurrence}
\end{prop}

\begin{proof}
For $a(n,k)$, we first note that a permutation of length $n$ must have at least zero ascents and no more than $n-1$ ascents.   There is one way to have no ascents (the decreasing permutation) and one way to have all $n-1$ possible ascents (the increasing permutation).

More generally, Lemma~\ref{L:blocks} shows that there are two ways for adjacent blocks to interact: 
\begin{center}
$\max(B_i)<\min(B_{i+1})$ or $\max(B_i) = \min(B_{i+1})+1$.
\end{center}
This first situation may occur no matter the sizes of blocks $B_i$ and $B_{i+1}$.  However, the second case may only happen if at least one of the blocks has size greater than 1; if both blocks have size 1 and $\max(B_i) = \min(B_{i+1})+1$ then $\max(B_i)$ and $\min(B_{i+1})$ form a descent and are actually in the same block.

Suppose that we wish to build a permutation of length $n$ with $k>0$ ascents.  Consider the permutation formed by the first $k$ blocks of the permutation, which has length $i$, ($1 \leq i \leq n-1$) and $k-1$ ascents.  There are two ways to add a new block of size $n-i$ and produce a two-pop-stack sortable permutation, with one exception: if $i=n-1$, and the permutation formed by the first $k$ blocks ends in a block of size 1, then there is a unique way to add a final block of size 1.  The $2 \sum_{i=1}^{n-1} a(i,k-1)$ term reflects the fact that there are generally two ways to add a final block to achieve a permutation of length $n$ with $k$ ascents.  The $b(n-1,k-1)$ term subtracts off the number of permutations for which there was only one way to add a final block to achieve a permutation of length $n$ with $k$ ascents.

The argument for $b(n,k)$ is similar.  Since permutations counted by $b(n,k)$ end in a block of size 1, the only way to have no ascents is for $n=1$ and $k=0$, which is covered in the $k=n-1$ case.  Then, as before a permutation of length $n$ cannot have less than zero ascents and can have no more than $n-1$ ascents.  There is still one way to have all $n-1$ possible ascents (the increasing permutation).

More generally, suppose that we wish to build a permutation of length $n$ with $k>0$ ascents and that ends in a block of size 1.  Consider the permutation formed by the first $k$ blocks of the permutation, which has length $n-1$.  There are two ways to add a new block of size 1 to produce a permutation of length $n$, unless the last block of the permutation on $n-1$ letters already ended in a block of size 1.
\end{proof}

Proposition~\ref{P:recurrence} implies following result:

\begin{theorem}
$$\sum_{\pi \in \mathcal{P}_{2}}x^{\left|\pi\right|}y^{\mathrm{asc}(\pi)} = \sum_{n=0}^\infty \sum_{k=0}^{n-1} a(n,k) x^ny^k= \dfrac{1-xy-x^2y+x^3y-2x^3y^2}{1-x-xy-x^2y-2x^3y^2}$$
\label{T:mgf}
\end{theorem}

\begin{proof}
Consider $$A(x,y)=\sum_{n \geq 0} \sum_{k \geq 0} a(n,k) x^ny^k$$ and $$B(x,y)=\sum_{n \geq 0} \sum_{k \geq 0} b(n,k) x^ny^k.$$  From the recurrences in Proposition~\ref{P:recurrence} we obtain $$A(x,y)=\frac{1}{1-x}+\frac{2xy}{1-x}\left(A(x,y)-1\right)-xyB(x,y)$$ and $$B(x,y)=x+x^2y+\frac{2x^3y}{1-x}+2xy\left(A(x,y)-\frac{1}{1-x}\right)-xy\left(B(x,y)-x\right).$$  Solving this system for $A(x,y)$ yields the bivariate generating function in the theorem.
\end{proof}

Expanding the generating function in Theorem~\ref{T:mgf} to see low-order terms, we have:
\begin{align*}
\dfrac{1-xy-x^2y+x^3y-2x^3y^2}{1-x-xy-x^2y-2x^3y^2}=&1+x+(y+1)x^2+(y^2+4y+1)x^3\\
&+(y^3+8y^2+6y+1)x^4\\
&+(y^4+12y^3+20y^2+8y+1)x^5\\
&+(y^5+16y^4+48y^3+36y^2+10y+1)x^6\\
&+\cdots
\end{align*}

We will return to the sequences from Proposition~\ref{P:recurrence} in Section~\ref{S:more_rec} as well as utilize the ascent structure of these sortable permutations in Section~\ref{S:poly}.  First however, notice plugging in $y=1$ yields the enumeration of two-pop-stack sortable permutations.
 
\begin{cor}
$$\sum_{\pi \in \mathcal{P}_{2}}x^{\left|\pi\right|} = \dfrac{1-x-x^2-x^3}{1-2x-x^2-2x^3}$$
\label{C:enum}
\end{cor}

This generating function corresponds to sequence A224232 in the On-Line Encyclopedia of Integer Sequences \cite{OEIS}.  From this rational generating function, we see that the number of two-pop-stack sortable permutations follows a linear recurrence with constant coefficients; that is,
\begin{equation}
\left|\mathcal{P}_{2,n}\right| = 2\left|\mathcal{P}_{2,n-1}\right|+\left|\mathcal{P}_{2,n-2}\right|+2\left|\mathcal{P}_{2,n-3}\right|.
\label{E:linrec}
\end{equation}

Equation~\ref{E:linrec} can also be found more simply, but without the refinement obtained by counting the number of ascents, by considering the placement of the $1$ in a two-pop-stack sortable permutation as shown below.

Let $I_n=1\cdots n$ be the increasing permutation of length $n$ and let $J_n = n\cdots 1$ be the decreasing permutation of length $n$.  Let $J_n^{(+k)} = (n+k)\cdots (1+k)$ be the decreasing permutation of length $n$ where all digits have been incremented by $k$.  Then, we can decompose the set $\mathcal{P}_{2,n}$ as described in Theorem~\ref{T:prefix}.

\begin{theorem}\label{T:prefix}
Suppose $\pi \in \mathcal{S}_n$ where $n \geq 4$.  Then $\pi \in \mathcal{P}_{2,n}$ if and only if one of the following is true:
\begin{enumerate}
\item $\pi = 1 \oplus \hat{\pi}$ where $\hat{\pi} \in \mathcal{P}_{2,n-1}$,
\item $\pi_i=1$ for some $i \geq 2$ where $\pi_1 \cdots \pi_{i-1}$ is the longest decreasing prefix of $\pi_1\cdots \pi_{i-1}\pi_{i+1}\cdots \pi_n$, and $\hat{\pi} = \operatorname{red}(\pi_1\cdots \pi_{i-1}\pi_{i+1}\cdots \pi_n) \in \mathcal{P}_{2,n-1}$,
\item $\pi_1\pi_2\pi_3=312$ and $\hat{\pi}=\operatorname{red}(\pi_{4}\cdots \pi_n) \in \mathcal{P}_{2,n-3}$.
\item $\pi_1\pi_2\pi_3=413$, $\pi_{4}\cdots \pi_n$ begins with a decreasing prefix of length at least 2 that ends in the digit 2, and  $\hat{\pi}=\operatorname{red}(\pi_{4}\cdots \pi_n) \in \mathcal{P}_{2,n-3}$,
\item $\pi=2\pi_21(\pi_2-1)\pi_5\cdots \pi_n$ where $\pi_5>\pi_2$ and $\operatorname{red}(\pi_2\pi_5\cdots \pi_n)\in \mathcal{P}_{2,n-3}$
\item $\pi=2\pi_2\cdots \pi_{i-1}1\pi_{i+1}\cdots \pi_n$ for some $i\geq 3$ where $\pi_2\cdots \pi_{i-1}$ is decreasing and $\operatorname{red}(\pi_2\cdots \pi_{i-1}\pi_{i+1}\cdots \pi_n)\in \mathcal{P}_{2,n-2}$, and if $i=3$, then $\pi_2<\pi_4$.
\end{enumerate}

\end{theorem}

For example, $\left|\mathcal{P}_{2,4}\right|=16$.  Here are the 16 permutations separated according to the six cases in Theorem~\ref{T:prefix}:
\begin{enumerate}
\item 1234, 1243, 1324, 1432, 1342, 1423 
\item 2134, 2143, 3142, 3214, 4213, 4321
\item 3124
\item (none)
\item 2413
\item 2314, 2431
\end{enumerate}

\begin{proof}
First, we claim that 1 must appear in the first two blocks of $\pi$.  Suppose to the contrary that the digit 1 appears in block 3 or later, and let $\pi^* = P(\pi)$.  If the first two blocks of $\pi$ have size 1, they will still be the first two blocks in $\pi^*$ and there will be an ascent between these blocks.  If either of the first two blocks of $\pi$ has size greater than 1, then in $\pi^*$ this block will be reversed to an increasing sequence.  Either way, there will be an ascent in $\pi^*$ before the digit 1.  On the other hand, since $\pi^*$ is one-pop-stack sortable, it must be the direct sum of decreasing permutations, so the first block of $\pi^* = J_i$ for some $i \geq 1$.  This means there cannot be an ascent in $\pi^*$ before the digit 1.  Therefore, the digit 1 must appear in the first two blocks of $\pi$.

Suppose 1 is in the first block of $\pi$, and consider the various sizes of the first block.

If the first block has size 1, then $\pi = 1 \oplus \hat{\pi}$ for some $\hat{\pi} \in \mathcal{P}_{2,n-1}$.  This is case 1.

If the first block has size $i \geq 2$, then $\pi_i=1$ and either $\pi_{i-1}<\pi_{i+1}$ or $\pi_{i-1}>\pi_{i+1}$.  $\pi_{i-1}<\pi_{i+1}$ is case 2.  Case 3 is $\pi_{i-1}>\pi_{i+1}$ where $\pi_{i-1}=3$ and case 4 is $\pi_{i-1}>\pi_{i+1}$ where $\pi_{i-1}=4$.  In both case 3 and case 4, notice that Lemma~\ref{L:blocks} implies that $i=2$ and $\left|B_2\right|=1$ since we have the case that $\max(B_1)=\min(B_2)+1$.  In case 4, the lemma further implies that $B_3$ consists of a decreasing sequence ending in 2.  If $\pi_{i-1}>\pi_{i+1}$ and $a=\pi_{i-1}>4$, then $\pi \notin \mathcal{P}_{2,n}$ since block 2 of $\pi^*$ would need to be equal to $J_{a-1}^{(+1)}$ but it is impossible to construct a decreasing subsequence of consecutive values of length four or more after one pass through a pop stack.

Finally, suppose 1 is in the second block of $\pi$.  Then by Lemma~\ref{L:blocks}, the maximum element of the first block is 2.  If the second block has size 2 and $\pi_2>\pi_4$, we are in case 5.  Otherwise, we are in case 6.
\end{proof}

Notice that cases 1 and 2 give two different ways to build a member of $\mathcal{P}_{2,n}$ from a member of $\mathcal{P}_{2,n-1}$.  Case 6 gives 1 way to build a member of $\mathcal{P}_{2,n}$ from any member of $\mathcal{P}_{2,n-2}$.  Case 3 gives 1 way to build a member of $\mathcal{P}_{2,n}$ from any member of $\mathcal{P}_{2,n-3}$.  Case 4 gives a way to build a member of $\mathcal{P}_{2,n}$ from any member of $\mathcal{P}_{2,n-3}$ that begins with a descent, and case 5 gives a way to build a member of $\mathcal{P}_{2,n}$ from any member of $\mathcal{P}_{2,n-3}$ that begins with an ascent.

Together, we have that $$\left|\mathcal{P}_{2,n}\right| = 2\left|\mathcal{P}_{2,n-1}\right|+\left|\mathcal{P}_{2,n-2}\right|+2\left|\mathcal{P}_{2,n-3}\right|.$$

The sequence obtained from this recurrence (A224232) also enumerates a different family of combinatorial objects as seen in Section~\ref{S:poly}.

\section{Polyominoes}\label{S:poly}

Both one-pop-stack sortable and two-pop-stack sortable permutations are in bijection with special families of polyominoes.  
Although the fact that these sets are equinumerous has been shown computationally, the bijections given in this section are new.  Moreover, they map ascents and descents of the appropriate permutations to nice features of the polyominoes.

Recall that a \emph{polyomino} is an edge-connected set of cells on the lattice $\mathbb{Z}^2$.  The size of a polyomino $P$ is the number of cells in $P$.  The polyominoes of size at most 3 are given in Figure~\ref{F:poly}.  In particular there is one polyomino of size 1, two of size 2, and six of size 3.  In general, the number of polyominoes of size $n$ for large $n$ remains an open problem.  However, we will consider a modified type of polyomino.

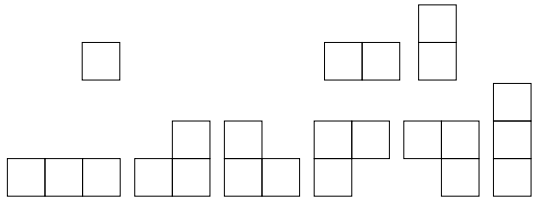
\begin{figure}[hbt]
\begin{center}
\scalebox{0.5}{
\begin{tikzpicture}
					\draw (0,0) rectangle (1,1);		
\end{tikzpicture}}
\hspace{1in}\scalebox{0.5}{
\begin{tikzpicture}
					\draw (0,0) rectangle (1,1);
					\draw (1,0) rectangle (2,1);		
\end{tikzpicture} \hspace{0.1in}
\begin{tikzpicture}
					\draw (0,0) rectangle (1,1);
					\draw (0,1) rectangle (1,2);		
\end{tikzpicture}}

\scalebox{0.5}{
\begin{tikzpicture}
					\draw (0,0) rectangle (1,1);
					\draw (1,0) rectangle (2,1);	
					\draw (2,0) rectangle (3,1);		
\end{tikzpicture}\hspace{0.1in}
\begin{tikzpicture}
					\draw (0,0) rectangle (1,1);
					\draw (1,0) rectangle (2,1);	
					\draw (1,1) rectangle (2,2);		
\end{tikzpicture}\hspace{0.1in}
\begin{tikzpicture}
					\draw (0,0) rectangle (1,1);
					\draw (1,0) rectangle (2,1);	
					\draw (0,1) rectangle (1,2);		
\end{tikzpicture}\hspace{0.1in}
\begin{tikzpicture}
					\draw (0,0) rectangle (1,1);
					\draw (0,1) rectangle (1,2);
				  \draw (1,1) rectangle (2,2);			
\end{tikzpicture}\hspace{0.1in}
\begin{tikzpicture}
					\draw (1,0) rectangle (2,1);
					\draw (0,1) rectangle (1,2);
				  \draw (1,1) rectangle (2,2);			
\end{tikzpicture}\hspace{0.1in}
\begin{tikzpicture}
					\draw (0,0) rectangle (1,1);
					\draw (0,1) rectangle (1,2);
				  \draw (0,2) rectangle (1,3);			
\end{tikzpicture}}

\end{center}

\caption{Small polyominoes in the plane}
\label{F:poly}
\end{figure}

Following \cite{AAB13} and \cite{AAB14}, we consider polyominoes on a twisted cylinder of width $w \in \mathbb{Z}^+$.  These polyominoes are drawn in the first quadrant of $\mathbb{Z}^2$ by identifying all pairs of cells with coordinates $(x,y)$ and $(x-w,y+1)$.  
Visually, instead of drawing polyominoes in the plane, we draw them on the surface shown in Figure~\ref{F:cylinder}.  Notice that this surface is a cylinder with a helix wrapped around it.  Vertical lines together with the helix partition the surface into cells.  If we begin at cell $(x,y)$ and move one cell to the right $w$ times, we end up in cell $(x,y+1)$, one cell above $(x,y)$.  Rather than drawing the twisted cylinder embedded in $\mathbb{R}^3$, we may visualize it in $\mathbb{R}^2$ as shown in Figure~\ref{F:flatcylinder}, where we show both the twisted cylinder of width 2 and the twisted cylinder of width 3.  With this convention, there are only four polyominoes of size 3 on a twisted cylinder of width 2; notice that 
\scalebox{0.2}{
\begin{tikzpicture}
					\draw (0,0) rectangle (1,1);
					\draw (1,0) rectangle (2,1);	
					\draw (2,0) rectangle (3,1);		
\end{tikzpicture}}, \scalebox{0.2}{\begin{tikzpicture}
					\draw (0,0) rectangle (1,1);
					\draw (1,0) rectangle (2,1);	
					\draw (0,1) rectangle (1,2);		
\end{tikzpicture}}, and \scalebox{0.2}{\begin{tikzpicture}
					\draw (1,0) rectangle (2,1);
					\draw (1,1) rectangle (2,2);	
					\draw (0,1) rectangle (1,2);		
\end{tikzpicture}}
 are all the same polyomino on the twisted cylinder of width 2 since they all cover cells 1, 2, and 3 in the appropriate part of Figure~\ref{F:flatcylinder}.  Polyominoes on twisted cylinders were introduced to find improved bounds on the number of polyominoes in the plane.  The polyominoes on width 2 and width 3 cylinders also are in bijection with one-pop-stack and two-pop-stack sortable permutations in a natural way.

\begin{figure}
\begin{center}
\begin{tikzpicture}
\begin{axis}[axis lines=none, xtick=\empty, ytick=\empty,clip=false,
zmin=0, zmax=6*pi,
xmin=-2,xmax=2,
ymin=-2,ymax=2,
width=7cm,height=7cm, view={-90}{-30},colormap = {whiteblack}{color(0cm)  = (white);color(1cm) = (black)}]
     \addplot3[surf,opacity = 0.99,color=gray, samples=50,domain=0:6*pi] 
  	({cos(deg(y))},{sin(deg(y))},{x});
  	\addplot3+[no markers,line width=2pt,color=black,samples y=0,domain=0:pi]
  ({sin(deg(x))},
   {cos(deg(x))},
   {x});
    \addplot3+[no markers,line width=2pt,color=black,samples y=0,domain=2*pi:3*pi]
  ({sin(deg(x))},
   {cos(deg(x))},
   {x});
   \addplot3+[no markers,line width=2pt,color=black,samples y=0,domain=4*pi:5*pi]
  ({sin(deg(x))},
   {cos(deg(x))},
   {x});
   \addplot3+[no markers,line width=2pt,color=black,samples y=0,domain=0:6*pi]
  ({1)},
   {0)},
   {x});
  \addplot3+[no markers,line width=2pt,color=black,samples y=0,domain=0:6*pi]
  ({0)},
   {1)},
   {x});
  \addplot3+[no markers,line width=2pt,color=black,samples y=0,domain=0:6*pi]
  ({0)},
   {-1)},
   {x});
  \addplot3+[no markers,line width=2pt,color=black,samples y=0,domain=0:6*pi]
  ({sqrt(2)/2)},
   {sqrt(2)/2)},
   {x});
   \addplot3+[no markers,line width=2pt,color=black,samples y=0,domain=0:6*pi]
  ({sqrt(2)/2)},
   {-sqrt(2)/2)},
   {x});
\end{axis}
\end{tikzpicture}
\end{center}
\caption{A twisted cylinder}
\label{F:cylinder}
\end{figure}
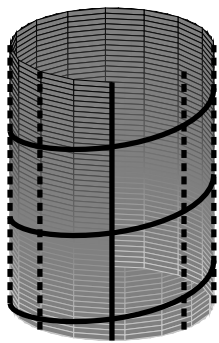

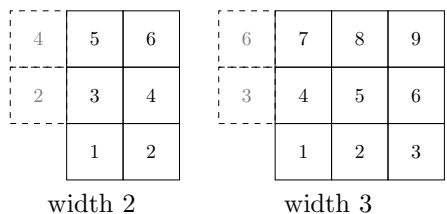
\begin{figure}
\begin{center}
\begin{tabular}{cc}
\scalebox{0.75}{
\begin{tikzpicture}
					\draw (0,0) rectangle (1,1);
					\draw (1,0) rectangle (2,1);
					\draw (0.5, 0.5) node {1};
					\draw (1.5, 0.5) node {2};
					\draw[dashed] (-1,1) rectangle (0,2);
					\draw (0,1) rectangle (1,2);
					\draw (1,1) rectangle (2,2);
					\draw[gray] (-0.5, 1.5) node {2};
					\draw (0.5, 1.5) node {3};
					\draw (1.5, 1.5) node {4};
					\draw[dashed] (-1,2) rectangle (0,3);
					\draw (0,2) rectangle (1,3);
					\draw (1,2) rectangle (2,3);
					\draw[gray] (-0.5, 2.5) node {4};
					\draw (0.5, 2.5) node {5};
					\draw (1.5, 2.5) node {6};
\end{tikzpicture}}&
\scalebox{0.75}{
\begin{tikzpicture}
					\draw (0,0) rectangle (1,1);
					\draw (1,0) rectangle (2,1);
					\draw (2,0) rectangle (3,1);
					\draw (0.5, 0.5) node {1};
					\draw (1.5, 0.5) node {2};
					\draw (2.5, 0.5) node {3};					
					\draw[dashed] (-1,1) rectangle (0,2);
					\draw (0,1) rectangle (1,2);
					\draw (1,1) rectangle (2,2);
					\draw (2,1) rectangle (3,2);
					\draw[gray] (-0.5, 1.5) node {3};
					\draw (0.5, 1.5) node {4};
					\draw (1.5, 1.5) node {5};
					\draw (2.5, 1.5) node {6};
					\draw[dashed] (-1,2) rectangle (0,3);
					\draw (0,2) rectangle (1,3);
					\draw (1,2) rectangle (2,3);
					\draw (2,2) rectangle (3,3);
					\draw[gray] (-0.5, 2.5) node {6};
					\draw (0.5, 2.5) node {7};
					\draw (1.5, 2.5) node {8};
					\draw (2.5, 2.5) node {9};
\end{tikzpicture}}\\
width 2&width 3\\
\end{tabular}
\end{center}
\caption{Twisted cylinders of width 2 and width 3}
\label{F:flatcylinder}
\end{figure}

From Avis and Newborn \cite{AN81} $\pi$ is one-pop-stack sortable if and only if $\pi$ is layered and there are $2^{n-1}$ such permutations of length $n$.  Similarly, Aleksandrowicz, Asinowski, and Barequet \cite{AAB13} observe that there are $2^{n-1}$ polyominoes of size $n$ on a twisted cylinder of width 2.  We reproduce the result that these permutations and polyominoes have the same enumeration via a bijection that preserves an additional property for each set.

\begin{theorem}~\label{T:bijection_1}
The one-pop-stack sortable permutations of length $n$ are in bijection with the polyominoes of size $n$ on a twisted cylinder of width 2.  Moreover, there are the same number of polyominoes of size $n$ on a twisted cylinder of width 2 with $k+1$ squares without an adjacent square to their right as there are one-pop-stack sortable permutations of length $n$ with $k$ ascents, namely $\binom{n-1}{k}$.
\end{theorem}

\begin{proof}
The bijection is as follows.  Consider a layered permutation $\pi$ of length $n$.  Let $b_i$ be the length of block $i$ of $\pi$.  For each $b_i$, construct a $1 \times b_i$ rectangular polyomino.  Place these rectangles on a strip of height 1, leaving one empty square between each rectangle.  Wrap the resulting strip around the twisted cylinder of width 2.  

Here, descents (i.e. adjacent letters in the same block of $\pi$) correspond to left-right adjacent pairs of squares in the corresponding polyomino.  Ascents correspond to squares of the polyomino with no square to their right.  Additionally, the last square of the polyomino cannot have a square to its right.  
\end{proof}

In Figure~\ref{F:width2} we see the permutation 4321657(10)98.  In this case $b_1=4$, $b_2=2$, $b_3=1$, and $b_4=3$.  We construct rectangular polyominoes of widths 4, 2, 1, and 3 and wrap them around the helix of width 2, leaving an empty square between each adjacent pair of rectangles.  

\begin{figure}
\begin{center}
\begin{tabular}{ccc}
\scalebox{0.4}{
\begin{tikzpicture}
					\draw (0,0) rectangle (10,10);		
					\fill[black] (.5,3.5) circle (.2cm);
					\fill[black] (1.5,2.5) circle (.2cm);
					\fill[black] (2.5,1.5) circle (.2cm);
					\fill[black] (3.5,0.5) circle (.2cm);
					\fill[black] (4.5,5.5) circle (.2cm);
					\fill[black] (5.5,4.5) circle (.2cm);
					\fill[black] (6.5,6.5) circle (.2cm);
					\fill[black] (7.5,9.5) circle (.2cm);
					\fill[black] (8.5,8.5) circle (.2cm);
					\fill[black] (9.5,7.5) circle (.2cm);
					\draw (4,0) -- (4,4);
					\draw (0,4) -- (4,4);	
					\draw (4,4) -- (4,6);
					\draw (4,4) -- (6,4);
					\draw (6,6) -- (4,6);
					\draw (6,6) -- (6,4);
					\draw (6,6) -- (7,6);
					\draw (6,6) -- (6,7);
					\draw (7,7) -- (7,6);
					\draw (7,7) -- (6,7);
					\draw (7,7) -- (7,10);
					\draw (7,7) -- (10,7);
\end{tikzpicture}}&
\hspace{0.5in}&
\scalebox{0.6}{\begin{tikzpicture}
					\fill[black]  (0,0) rectangle (1,1);
					\fill[black]  (1,0) rectangle (2,1);	
					\fill[black]  (0,1) rectangle (1,2);
					\fill[black] (1,1) rectangle (2,2);
					\draw  (0,2) rectangle (1,3);
					\fill[gray]  (1,2) rectangle (2,3);
					\fill[gray] (0,3) rectangle (1,4);
					\draw  (1,3) rectangle (2,4);
					\fill[black] (0,4) rectangle (1,5);
					\draw  (1,4) rectangle (2,5);
					\fill[gray]  (0,5) rectangle (1,6);
					\fill[gray]  (1,5) rectangle (2,6);
					\fill[gray]  (0,6) rectangle (1,7);
					\draw (1,6) rectangle (2,7);
					\end{tikzpicture}}
\end{tabular}

\vspace{0.2in}

\scalebox{0.5}{
\begin{tikzpicture}
\draw[step=1cm,color=gray,dashed] (0,0) grid (13,1);
					\fill[black] (0,0) rectangle (1,1);
					\fill[black] (1,0) rectangle (2,1);	
					\fill[black] (2,0) rectangle (3,1);
					\fill[black] (3,0) rectangle (4,1);
					\fill[gray] (5,0) rectangle (6,1);
					\fill[gray] (6,0) rectangle (7,1);	
					\fill[black] (8,0) rectangle (9,1);
					\fill[gray] (10,0) rectangle (11,1);
					\fill[gray] (11,0) rectangle (12,1);	
					\fill[gray] (12,0) rectangle (13,1);
\end{tikzpicture}}

\end{center}
\caption{The one-pop-stack sortable permutation 4321657(10)98 and its corresponding polyomino}
\label{F:width2}
\end{figure}
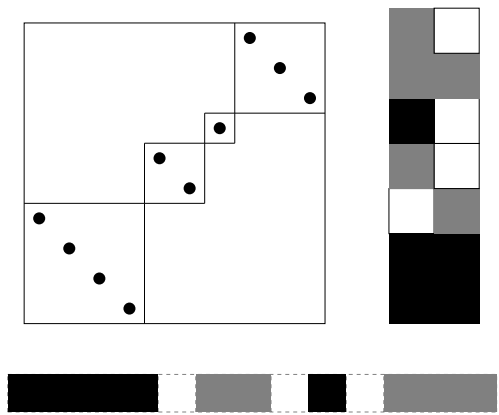

We showed in Corollary~\ref{C:enum} that two-pop-stack sortable permutations are counted by sequence A224232 in the On-Line Encyclopedia of Integer Sequences.  Aleksandrowicz, Asinowski, and Barequet \cite{AAB13} showed that polyominoes on a twisted cylinder of width 3 have this same enumeration.  They counted these polyominoes via a recurrence in cases, of a similar flavor to the proof of Theorem~\ref{T:prefix} and also gave a closed formula for the recurrence.

Aleksandrowicz, Asinowski, and Barequet \cite{AAB13} also found bijections between these polyominoes and three permutation classes containing eight patterns of length $4$.  

\begin{theorem}~\label{T:AAB} (Aleksandrowicz, Asinowski, and Barequet)
The classical permutation classes 
$$\Av(2431,3412,3421,4123,4213,4231,4312,4321),$$
$$\Av(2413,3412,3421,4123,4132,4213,4312,4321),\text{ and }$$
$$\Av(1342,1432,3142,3412,4123,4132,4213,4312)$$ 
are Wilf equivalent to each other and are in bijection with the polyominoes on a twisted cylinder of width 3.
\end{theorem}

The bijection to be given in this paper between these polyominoes and the two-pop-stack sortable permutations is based on the block structure permutations need to be sortable.  Moreover, while the sortable permutations studied here are also classified by the avoidance of eight patterns, they do not form a classical permutation class. That is, subpermutations of the sortable permutations are not guaranteed to be sortable.  As such, we obtain a Wilf equivalence between  classical permutation classes and the non-classical permutation class introduced in Theorem~\ref{T:patterns}.  

For the next theorem, recall $a(n,k)$ is the number of two-pop-stack sortable permutations of length $n$ with $k$ ascents.

\begin{theorem}~\label{T:bijection_2}
The two-pop-stack sortable permutations of length $n$ are in bijection with the polyominoes of size $n$ on a twisted cylinder of width 3.  Moreover, there are $a(n,k)$ polyominoes of size $n$ on a twisted cylinder of width 3 with $k+1$ squares without an adjacent square to its right.
\end{theorem}

\begin{proof}
As in the bijection for one-pop-stack sortable permutations, descents of $\pi$ are sent to left-right adjacent pairs of squares in the corresponding polyomino and ascents are sent to squares with no square to the right.  And again the last square of the polyomino will have no square to its right.

Consider a two-pop-stack sortable permutation $\pi$.  Let $b_i$ be the length of block $i$ of $\pi$.  For each $b_i$, construct a $1 \times b_i$ rectangular polyomino.  As before, we place these rectangles on a strip of height 1 with an extra consideration.  For blocks $B_i$ and $B_{i+1}$, by Lemma~\ref{L:blocks}, either $\max(B_i)<\min(B_{i+1})$ or $\max(B_i) = \min(B_{i+1})+1$.  The first case may happen no matter the size of the blocks, but $\max(B_i) = \min(B_{i+1})+1$ requires that at least one of the blocks has size greater than 1.  Accordingly, if $\max(B_i)<\min(B_{i+1})$, the corresponding $1\times b_i$ and $1\times b_{i+1}$ rectangles should have two empty squares between them.  This guarantees that the last square in the $1 \times b_i$ rectangle is below the first square in the $1 \times b_{i+1}$ rectangle.  On the other hand, if $\max(B_i) = \min(B_{i+1})+1$, the corresponding $1\times b_i$ and $1\times b_{i+1}$ rectangles should have one empty square between them.  Since at least one of the blocks has size greater than one, the two blocks still form part of a connected polyomino.  Wrap the resulting strip around the twisted cylinder of width 3. 
\end{proof} 

In Figure~\ref{F:width3} we see the permutation 64321587(12)(10)9(14)(13)(11).  In this case $b_1=5$, $b_2=1$, $b_3=2$, $b_4=3$, and $b_5=3$.  We construct rectangular polyominoes of widths 5, 1, 2, 3, and 3.  Blocks 1 and 2 as well as blocks 4 and 5 have $\max(B_i) = \min(B_{i+1})+1$ so we leave one empty square between the corresponding rectangles.  Blocks 2 and 3 as well as blocks 3 and 4 have $\max(B_i)<\min(B_{i+1})$ so we leave two empty squares between the corresponding rectangles.  Then, we wrap the resulting strip of separated rectangles around the twisted cylinder of width 3.

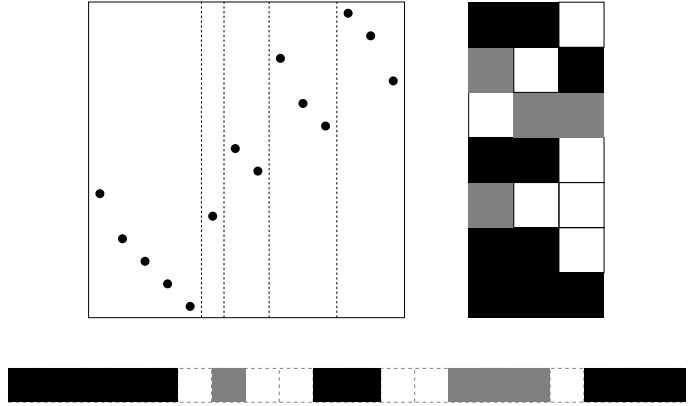
\begin{figure}
\begin{center}
\begin{tabular}{ccc}
\scalebox{0.3}{
\begin{tikzpicture}
					\draw (0,0) rectangle (14,14);		
					\fill[black] (.5,5.5) circle (.2cm);
					\fill[black] (1.5,3.5) circle (.2cm);
					\fill[black] (2.5,2.5) circle (.2cm);
					\fill[black] (3.5,1.5) circle (.2cm);
					\fill[black] (4.5,0.5) circle (.2cm);
					\fill[black] (5.5,4.5) circle (.2cm);
					\fill[black] (6.5,7.5) circle (.2cm);
					\fill[black] (7.5,6.5) circle (.2cm);
					\fill[black] (8.5,11.5) circle (.2cm);
					\fill[black] (9.5,9.5) circle (.2cm);
					\fill[black] (10.5,8.5) circle (.2cm);
					\fill[black] (11.5,13.5) circle (.2cm);
					\fill[black] (12.5,12.5) circle (.2cm);
					\fill[black] (13.5,10.5) circle (.2cm);
					\draw[dashed] (5,0) -- (5,14);
					\draw[dashed] (6,0) -- (6,14);
					\draw[dashed] (8,0) -- (8,14);
					\draw[dashed] (11,0) -- (11,14);

\end{tikzpicture}}&
\hspace{0.5in}&
\scalebox{0.6}{\begin{tikzpicture}
					\fill[black]  (0,0) rectangle (1,1);
					\fill[black]  (1,0) rectangle (2,1);
					\fill[black]  (2,0) rectangle (3,1);
					\fill[black]  (0,1) rectangle (1,2);
					\fill[black]  (1,1) rectangle (2,2);
					\draw  (2,1) rectangle (3,2);
					\fill[gray]  (0,2) rectangle (1,3);
					\draw  (1,2) rectangle (2,3);
					\draw  (2,2) rectangle (3,3);
					\fill[black]  (0,3) rectangle (1,4);
					\fill[black]  (1,3) rectangle (2,4);
					\draw  (2,3) rectangle (3,4);
					\draw  (0,4) rectangle (1,5);
					\fill[gray]  (1,4) rectangle (2,5);
					\fill[gray]  (2,4) rectangle (3,5);
					\fill[gray]  (0,5) rectangle (1,6);
					\draw (1,5) rectangle (2,6);
					\fill[black]  (2,5) rectangle (3,6);
					\fill[black]  (0,6) rectangle (1,7);
					\fill[black] (1,6) rectangle (2,7);
					\draw  (2,6) rectangle (3,7);
					\end{tikzpicture}}
\end{tabular}

\vspace{0.2in}

\scalebox{0.45}{
\begin{tikzpicture}
\draw[step=1cm,color=gray,dashed] (0,0) grid (20,1);
					\fill[black] (0,0) rectangle (1,1);
					\fill[black] (1,0) rectangle (2,1);	
					\fill[black] (2,0) rectangle (3,1);
					\fill[black] (3,0) rectangle (4,1);
					\fill[black] (4,0) rectangle (5,1);
					\fill[gray] (6,0) rectangle (7,1);	
					\fill[black] (9,0) rectangle (10,1);
					\fill[black] (10,0) rectangle (11,1);
					\fill[gray] (13,0) rectangle (14,1);
					\fill[gray] (14,0) rectangle (15,1);	
					\fill[gray] (15,0) rectangle (16,1);
					\fill[black] (17,0) rectangle (18,1);
					\fill[black] (18,0) rectangle (19,1);	
					\fill[black] (19,0) rectangle (20,1);
\end{tikzpicture}}

\end{center}
\caption{The two-pop-stack sortable permutation 64321587(12)(10)9(14)(13)(11) and its corresponding polyomino}
\label{F:width3}
\end{figure}

The following Wilf equivalence is an immediate consequence of Theorems~\ref{T:patterns}, ~\ref{T:AAB}, and ~\ref{T:bijection_2}. 

\begin{cor} 
The non-classical permutation class \\ $\Av(2341, 3412, 3421,4123, 4231, 4312, 4\overline{1}352, 413\overline{5}2)$ is Wilf equivalent to 
$$\Av(2431,3412,3421,4123,4213,4231,4312,4321),$$
$$\Av(2413,3412,3421,4123,4132,4213,4312,4321),\text{ and }$$
$$\Av(1342,1432,3142,3412,4123,4132,4213,4312)$$ 
\end{cor}

Despite the naturalness of the bijection given in Theorem~\ref{T:bijection_1} and Theorem~\ref{T:bijection_2}, it turns out that these are not two special cases of a more general phenomenon.  One might conjecture that three-pop-stack sortable permutations are in bijection with polyominoes on a twisted cylinder of width 4.  However, all 24 permutations of length 4 are $m$-pop-stack sortable when $m \geq 3$ and there are only 19 polyominoes of size 4 in the plane (and thus on a a twisted cylinder of width $w \geq 4$).  The generalization of the bijections provided does however give an injection from the set of polyominoes on a twisted cylinder of width $w=m+1$ to the set of $m$-pop-stack sortable permutations for $m \geq 3$.

The enumeration of $m$-pop-stack-sortable permutations for $m \geq 3$ is considered in a recent paper of Claesson and Gu$\eth$mundsson \cite{CG17}, without bijective correspondences.

\section{Basic properties and conjectures on sortable permutations classified by ascents}~\label{S:more_rec}

The triangles formed by $\{a(n,k)\}$ (the number of two-pop-stack sortable permutations of length $n$ with $k$ ascents) and $\{b(n,k)\}$ (the number of two-pop-stack sortable permutations of length $n$ with $k$ ascents where the last block has size $1$) from Proposition~\ref{P:recurrence} have some other interesting properties in their own right.  These coefficients characterize the sortable permutations partially in terms of partitions, so it may not be surprising there are some parallels with the more general notion of partitions.

For example, if $\pi$ is two-pop-stack sortable, has length $n \geq 3$, and $n-2$ ascents, then $\pi$ has $n-1$ blocks.  There are four ways to structure the one block of size two (relative to the blocks before and after it) if the block is an interior block and only two ways to structure the block of size two (relative to the one block adjacent to it) if it is the first or last block.  Hence $a(n,n-2) = 4(n-2)$.  Lemma~\ref{L:basic} lists several small cases that can also be handled relatively easily and also shows a symmetry found in the $b(n,k)$ case.

\begin{lemma}~\label{L:basic}
We have the following identities:
\begin{enumerate}
\item $a(n,n-2) = 4(n-2)$ for $n \geq 3$.
\item $a(n,1) = 2(n-2)$ for $n \geq 4$.
\item $a(n,2) = 2n(n-3)$ for $n \geq 4$.
\item $b(n,n-2) = b(n,2)= 4(n-3)+2$ for $n \geq 4$.
\end{enumerate}
\end{lemma}

Based on our data, it also appears the sequences $\{a(n,k)\}$ and $\{b(n,k)\}$ may have some other nice features:

\begin{conjecture}  The sequence $\{a(n,k)\}_{k=0}^{n}$ is log-concave for all $n$.
\end{conjecture}

\begin{conjecture}  The sequence $\{a(2n+1,n)\}$ is enumerated by 
\[
\sum_{i=0}^{n-1}(-1)^i \binom{2n-2i}{n-i}\binom{n-1}{i} 2^{n-i} \quad \text{ for }n \geq 1
\]
and has generating function
\[
\sqrt{\frac{1+x}{1-7x}}.
\]
This formula and generating function appear in OEIS~\cite{OEIS} for sequence A085458.
\end{conjecture}

\begin{conjecture}
The set of two-pop-stack sortable permutations of length $2n+1$ with exactly $n$ ascents has an equal number of permutations with last block of size one as permutations with last block size greater than one.  That is, $a(2n+1,n)=2b(2n+1,n)$ for all $n$.
\end{conjecture}

\end{document}